\documentclass[12pt]{article}

\usepackage{amsmath, amssymb, amsfonts, amsbsy, amscd, latexsym, amsthm}
\usepackage{enumitem}
\date{}

\newtheorem{Theorem}{Theorem}[section]
\newtheorem{Proposition}[Theorem]{Proposition}

\theoremstyle{definition}

\theoremstyle{remark}
\newtheorem{Remark}[Theorem]{Remark}

\numberwithin{equation}{section}

\title{Relations for a class of terminating ${}_4F_3(4)$ hypergeometric series}

\author{Ilia D. Mishev
\footnote{Department of Mathematics, University of Colorado Boulder,
Campus Box 395, Boulder, CO 80309-0395, U.S.A.
E-mail address: ilia.mishev@colorado.edu}}

\begin{document}

\maketitle

\begin{abstract}
We derive relations for a certain class of terminating
${}_4F_3(4)$ hypergeometric series
with three free parameters. The invariance
group composed of these relations is shown to be isomorphic
to the symmetric group $S_3$. We further study relations for 
terminating ${}_3F_2(4)$ series that fall under two families.
By using a series reversal, we examine the corresponding 
terminating ${}_4F_3(1/4)$ and ${}_3F_2(1/4)$ series
relations. We additionally derive formulas for the sums of
the first $n+1$ terms of several nonterminating 
${}_3F_2(4)$ and ${}_3F_2(1/4)$
series. We also show
how certain known summation formulas for terminating
${}_2F_1(4)$ and ${}_3F_2(4)$ series follow as limiting
cases of some of our relations.
\end{abstract}

\section{Introduction}

The theory of the
hypergeometric series of type ${}_2F_1$
was systematically developed by Gauss \cite{Ga}.
Subsequently, generalized hypegeometric series
of type ${}_pF_q$, where $p$ and $q$ are nonnegative
integers, 
and special summations and relations among such series,
were investigated in the late nineteenth and   
early twentieth century by Thomae \cite{T}, 
Barnes \cite{Bar1,Bar2},
Ramanujan (see \cite{Har}), 
Whipple \cite{Wh1,Wh2,Wh3}, 
Bailey \cite{Ba1,Ba}, and others. 

Over the last thirty-five years there has
been a renewed interest in studying
hypergeometric series. In particular, among other things,
relations involving 
hypergeometric and basic hypergeometric series
have been described in terms of group theory
frameworks in papers by Beyer, Louck, and Stein \cite{BLS},
Srinivasa Rao, Van der Jeugt,  Raynal, Jagannathan, and Rajeswari
\cite{RJRJR}, Formichella, Green, and Stade \cite{FGS},
Mishev \cite{M1}, Green, Mishev, and Stade \cite{GMS,GMS2},
Van der Jeugt and Srinivasa Rao \cite{JR}, 
Lievens and Van der Jeugt \cite{LV1,LV2}.
Other works include 
Groenevelt \cite{Groen}, van de Bult, Rains, and Stokman \cite{BRS}, 
Krattenthaler and Rivoal \cite{KR}.

There are numerous applications of
hypergeometric series. There have been recent papers by 
Bump \cite{Bu}, Stade \cite{St1,St2,St3,St4}, 
and Stade and Taggart \cite{ST}
with applications in the theory of automorphic functions. 
Some other recent works, with applications in physics, were written by
Drake \cite{Dra}, Grozin \cite{Groz}, and Raynal \cite{R}.

In this paper, we study a certain class of terminating 
${}_4F_3(4)$ hypergeometric series 
with three free parameters
(see Section 2 for the relevant definitions
and terminology pertaining to hypergeometric series).
Hypergeometric series with argument 4 have not been
studied much in the past. Some of the few previous works
are papers by Chu \cite{Chu1} studying certain terminating
${}_2F_1(4)$ series summations perturbed by
two integer parameters, and Chen and Chu 
\cite{ChenChu1,ChenChu2} examining two classes of 
terminating ${}_3F_2(4)$ series summations perturbed by
two and three integer parameters, respectively.

The family of terminating ${}_4F_3(4)$ series we consider is given by
\begin{equation}
\label{Ser}
{}_4F_3 \left( \left. 
{\displaystyle -n,\frac{a}{2},\frac{a+1}{2},b
\atop \displaystyle a,1+a-c,c}\right| 
4\right),\\
\end{equation}
where $n$ is a nonnegative integer and the complex numbers
$a,b$, and $c$ are the three free parameters.

In Section 3, we employ of method of Bailey's (see 
\cite[Section 4.3]{Ba}) along with the
Chu--Vandermonde formula (see (\ref{chv})) to obtain a
transformation of the series in (\ref{Ser}) into a 
terminating ${}_4F_3(1/4)$ series. By reversing
the order of summation of the latter series
(see (\ref{rs})), 
we obtain a transformation of
(\ref{Ser}) into another terminating ${}_4F_3(4)$
series of the same family. By suitably 
normalizing that transformation, we obtain
a group of six invariance relations. Furthermore, the invariance
group composed of these six relations is shown to be isomorphic
to the symmetric group $S_3$.

In Section 4, we study terminating ${}_3F_2(4)$ series relations that 
follow from our terminating ${}_4F_3(4)$ series relations in Section 3.
In particular, we look at the two different families of 
terminating ${}_3F_2(4)$ series given by
\begin{equation}
\label{Ser3F21}
{}_3F_2 \left( \left. 
{\displaystyle -n,\frac{a}{2},\frac{a+1}{2}
\atop \displaystyle 1+a-c,c}\right| 
4\right)\\
\end{equation}
and
\begin{equation}
\label{Ser3F22}
{}_4F_3 \left( \left. 
{\displaystyle -n,\frac{a}{2},\frac{a+1}{2}
\atop \displaystyle a,c}\right| 
4\right).\\
\end{equation}
The first of these two families, under a suitable
normalization, is shown to again have an invariance
group isomorphic to the symmetric group $S_3$, while 
we do not have an $S_3$-symmetry in the second one of
these families. We still have two nontrivial
relations for the family (\ref{Ser3F22}), and
one of those two relations is relevant to the works
of Chu \cite{Chu1} and Chen and Chu \cite{ChenChu1}.
In fact, by taking certain limits of that 
relation, we can obtain some of the formulas found by
Chu in \cite{Chu1} and Chen and Chu in \cite{ChenChu1}.
We further obtain formulas for the sums of the first $n+1$ terms 
of certain divergent ${}_3F_2(4)$ hypergeometric series.

By series reversal, we explore in Section 5 corresponding 
relations among terminating ${}_4F_3(1/4)$ series, and 
the corresponding invariance group isomorphic to $S_3$.

Finally, in Section 6, we study, as consequence of our relations in
Section 5, the different types of relations among 
terminating ${}_3F_2(1/4)$
hypergeometric series which fall under two families. We also
give a formula for the sum of the first $n+1$ terms of a certain
nonterminating ${}_3F_2(1/4)$ series.

\section{Preliminaries}

The hypergeometric series of type ${}_{p}F_q$ is 
the power series in $z$ defined by
\begin{equation}
\label{hsd}
{}_{p}F_q \left( \left.{\displaystyle a_1,a_2,\ldots,a_{p}
\atop \displaystyle b_1,b_2,\ldots,b_q}\right| z\right) =
\sum_{n=0}^{\infty} \frac{(a_1)_n(a_2)_n \cdots
(a_{p})_n}{n!(b_1)_n(b_2)_n \cdots (b_q)_n}z^n,
\end{equation}
where $p$ and $q$ are nonnegative integers, $a_1,a_2,\ldots,a_{p},
b_1,b_2,\ldots,b_q, z\in \mathbb{C}$, and the rising factorial
$(a)_n$ is given by
\begin{equation*}
(a)_n=\left\{ \begin{array}{ll}
a(a+1)\cdots(a+n-1), & n>0,\\
1, & n=0.
\end{array} \right.
\end{equation*}

In this paper we will consider the case where $p=q+1$.
The series of type ${}_{q+1}F_q$ converges absolutely if $|z|<1$ or if
$|z|=1$ and $\textrm{Re}(\sum_{i=1}^qb_i-\sum_{i=1}^{q+1}a_i)>0$ (see
\cite[p.\ 8]{Ba}). We assume that no denominator parameter
$b_1,b_2,\ldots,b_q$ is a negative integer or zero. When a numerator
parameter $a_1,a_2,\ldots,a_{q+1}$ is a negative integer or zero, the
series has only finitely many nonzero terms and is said to 
{\it terminate}.

If $z=1$, we say that the series is of {\it unit argument} and of type
${}_{q+1}F_q(1)$. When $\sum_{i=1}^qb_i-\sum_{i=1}^{q+1}a_i=1$, the
series is called {\it Saalsch\"utzian}. If
$1+a_1=b_1+a_2=\cdots=b_q+a_{q+1}$, the series is called 
{\it well-poised}.
A well-poised series that satisfies $a_2=1+\frac{1}{2}a_1$ is called
{\it very-well-poised}.

We will use the classical Chu--Vandermonde formula
(see \cite[Section 1.3]{Ba}), which gives us the sum
of a terminating ${}_2F_1(1)$ series:
\begin{equation}
\label{chv}
{}_2F_1 \left( \left. 
{\displaystyle -n,a
\atop \displaystyle b}\right| 
1\right)
=\frac{(b-a)_n}{(b)_n}.
\end{equation}

If a hypergeometric series terminates, we can reverse its
order of summation. Using the readily verified identity
\begin{equation}
\label{nk}
(a)_{n-k}=\frac{(-1)^k (a)_n}{(1-a-n)_k},
\end{equation}
we can derive the following well-known 
``summation-reversal" formula:
\begin{eqnarray}
\label{rs}
&&{}_{p+1}F_q \left( \left. 
{\displaystyle -n,a_1,a_2,\ldots,a_p
\atop \displaystyle b_1,b_2,\ldots,b_q}\right| 
x\right)\\
&&=\frac{(a_1)_n(a_2)_n\cdots (a_p)_n(-x)^n}
{(b_1)_n(b_2)_n\cdots (b_q)_n}\nonumber\\
&&\times
{}_{q+1}F_p \left( \left. 
{\displaystyle -n,1-b_1-n,1-b_2-n,\ldots,1-b_q-n
\atop \displaystyle 1-a_1-n,1-a_2-n,\ldots,1-a_p-n}\right| 
\frac{(-1)^{p+q}}{x}\right).\nonumber
\end{eqnarray}
We will use Equation (\ref{rs}) in Sections 3 and 5
when expressing terminating ${}_4F_3(1/4)$
series in terms of
terminating ${}_4F_3(4)$ series and vice versa.

\section{Relations for the terminating
${}_4F_3(4)$ series}

In this section, we derive our relations for the
terminating ${}_4F_3(4)$ 
hypergeometric series and examine the structure
of those relations. The relations developed in this section
form the foundation for the rest of the paper.

We begin with a general
proposition that expresses a certain terminating
hypergeometric series as a sum of 
terminating hypergeometric series
of lower order:

\begin{Proposition}
\label{3P1}
If $n$ is a nonnegative integer, the following general identity holds:
\begin{eqnarray}
\label{1e3P1}
&&{}_{p+3}F_{q+3} \left( \left. 
{\displaystyle -n,\frac{a}{2},\frac{a+1}{2},a_1,a_2,\ldots,a_p
\atop \displaystyle a,1+a-c,c,b_1,b_2,\ldots,b_q}\right| 
x\right) \\
&&=\sum_{m=0}^n \left(
\frac{(-n)_m(a_1)_m(a_2)_m\cdots(a_p)_m\left(\frac{x}{4}\right)^m}
{m!(1+a-c)_m(b_1)_m(b_2)m\cdots(b_q)_m}\right.\nonumber\\
&&\times \left. 
{}_{p+1}F_{q+1} \left( \left. 
{\displaystyle -n+m,a_1+m,a_2+m,\ldots,a_p+m
\atop \displaystyle c,b_1+m,b_2+m,\ldots,b_q+m}\right| 
\frac{x}{4}\right) \right).\nonumber
\end{eqnarray} 
\end{Proposition}

\begin{proof}
By the Chu--Vandermonde formula, we have
\begin{eqnarray}
\label{1ep3P1}
&&{}_2F_1 \left( \left. 
{\displaystyle -n,1-c-n
\atop \displaystyle 1+a-c}\right| 
1\right)
=\frac{(a+n)_n}{(1+a-c)_n}\\
&&=\frac{(a)_{2n}}{(a)_n(1+a-c)_n}
=\frac{\left(\frac{a}{2}\right)_n\left(\frac{a+1}{2}\right)_n4^n}
{(a)_n(1+a-c)_n}.\nonumber
\end{eqnarray}
Using this in the summation expansion of
the left-hand side of (\ref{1e3P1}), we obtain
\begin{eqnarray}
\label{2ep3P1}
&&{}_{p+3}F_{q+3} \left( \left. 
{\displaystyle -n,\frac{a}{2},\frac{a+1}{2},a_1,a_2,\ldots,a_p
\atop \displaystyle a,1+a-c,c,b_1,b_2,\ldots,b_q}\right| 
x\right) \\
&&=\sum_{k=0}^n \left(
\frac{(-n)_k(a_1)_k(a_2)_k\cdots(a_p)_k\left(\frac{x}{4}\right)^k}
{k!(c)_k(b_1)_k(b_2)_k\cdots(b_q)_k}
{}_2F_1 \left( \left. 
{\displaystyle -k,1-c-k
\atop \displaystyle 1+a-c}\right| 
1\right)
\right)\nonumber\\
&&=\sum_{k=0}^n \left(
\frac{(-n)_k(a_1)_k(a_2)_k\cdots(a_p)_k\left(\frac{x}{4}\right)^k}
{k!(c)_k(b_1)_k(b_2)_k\cdots(b_q)_k}
\sum_{m=0}^k
\frac{(-k)_m(1-c-k)_m}{m!(1+a-c)_m}
\right)\nonumber\\
&&=\sum_{m=0}^n\sum_{k=m}^n 
\frac{(-n)_k(-1)^m(1-c-k)_m(a_1)_k(a_2)_k\cdots(a_p)_k\left(\frac{x}{4}\right)^k}
{m!(k-m)!(1+a-c)_m(c)_k(b_1)_k(b_2)_k\cdots(b_q)_k}\nonumber\\
&&=\sum_{m=0}^n\sum_{t=0}^{n-m} 
\frac{(-n)_{m+t}(a_1)_{m+t}(a_2)_{m+t}\cdots(a_p)_{m+t}\left(\frac{x}{4}\right)^{m+t}}
{m!t!(1+a-c)_m(c)_t(b_1)_{m+t}(b_2)_{m+t}\cdots(b_q)_{m+t}},\nonumber
\end{eqnarray}
where in our simplifications we have used
\begin{equation*}
\frac{(-k)_m}{k!m!}=\frac{(-1)^m}{m!(k-m)!}
\end{equation*}
and, after the change in index $k=m+t$,
\begin{eqnarray*}
\frac{(1-c-k)_m}{(c)_k}=\frac{(1-c-m-t)_m}{(c)_{m+t}}\\
=\frac{(-1)^m(c+t)_m}{(c)_{m+t}}
=\frac{(-1)^m}{(c)_t}.
\end{eqnarray*}

From here,
\begin{eqnarray}
\label{4ep3P1}
&&{}_{p+3}F_{q+3} \left( \left. 
{\displaystyle -n,\frac{a}{2},\frac{a+1}{2},a_1,a_2,\ldots,a_p
\atop \displaystyle a,1+a-c,c,b_1,b_2,\ldots,b_q}\right| 
x\right) \\
&&=\sum_{m=0}^n\left(
\frac{(-n)_{m}(a_1)_{m}(a_2)_{m}\cdots(a_p)_{m}\left(\frac{x}{4}\right)^{m}}
{m!(1+a-c)_m(b_1)_{m}(b_2)_{m}\cdots(b_q)_{m}}\right.\nonumber\\
&&\times\left.\sum_{t=0}^{n-m} 
\frac{(-n+m)_{t}(a_1+m)_{t}(a_2+m)_{t}\cdots(a_p+m)_{t}\left(\frac{x}{4}\right)^{t}}
{t!(c)_t(b_1+m)_{t}(b_2+m)_{t}\cdots(b_q+m)_{t}}\right).\nonumber
\end{eqnarray}
The right-hand side above is equal to the right-hand side of
(\ref{1e3P1}) and the proof is complete. 
\end{proof}

Proposition \ref{3P1} above is of similar nature to 
\cite[Eqs.\ (4.3.1) and (4.3.6)]{Ba}. We employ this 
proposition to derive a relation between a terminating 
${}_4F_3(4)$ series and a terminating ${}_4F_3(1/4)$ series:

\begin{Proposition}
\label{3P2}
If $n$ is a nonnegative integer, the following relation holds:
\begin{eqnarray}
\label{1e3P2}
&&{}_4F_3 \left( \left. 
{\displaystyle -n,\frac{a}{2},\frac{a+1}{2},b
\atop \displaystyle a,1+a-c,c}\right| 
4\right)\\
&&=\frac{(c-b)_n}{(c)_n}
{}_4F_3 \left( \left. 
{\displaystyle -n,1-c-n,b,1+b-c
\atop \displaystyle 1+a-c,\frac{1+b-c-n}{2},\frac{2+b-c-n}{2}}\right| 
\frac{1}{4}\right).\nonumber
\end{eqnarray}  
\end{Proposition}

\begin{proof}
Letting $p=1$, $q=0$, $a_1=b$, and $x=4$ in Proposition \ref{3P1},
we obtain
\begin{eqnarray}
\label{1ep3P2}
&&{}_{4}F_{3} \left( \left. 
{\displaystyle -n,\frac{a}{2},\frac{a+1}{2},b
\atop \displaystyle a,1+a-c,c}\right| 
4\right) \nonumber\\
&&=\sum_{m=0}^n \left(
\frac{(-n)_m(b)_m}
{m!(1+a-c)_m}
{}_{2}F_{1} \left( \left. 
{\displaystyle -n+m,b+m
\atop \displaystyle c}\right| 
1\right) \right).
\end{eqnarray} 

Summing the ${}_2F_1$ series in (\ref{1ep3P2}) by the
Chu--Vandermonde formula and simplifying gives
\begin{eqnarray}
\label{2ep3P2}
&&{}_{2}F_{1} \left( \left. 
{\displaystyle -n+m,b+m
\atop \displaystyle c}\right| 
1\right) 
=\frac{(c-b-m)_{n-m}}{(c)_{n-m}}\nonumber\\
&&=\frac{(-1)^n(1+b-c-n+2m)_{n-m}(1-c-n)_m}
{(c)_n}\nonumber\\
&&=\frac{(-1)^n(1+b-c-n)_{n+m}(1-c-n)_m}
{(c)_n(1+b-c-n)_{2m}}\nonumber\\
&&=\frac{(c-b)_n(1+b-c)_m(1-c-n)_m}
{(c)_n\left(\frac{1+b-c-n}{2}\right)_{m}\left(\frac{2+b-c-n}{2}\right)_{m}4^m}.
\end{eqnarray}

Combining (\ref{1ep3P2}) and (\ref{2ep3P2})
yields the result.
\end{proof}

Reversing the order of summation in the terminating
${}_4F_3(1/4)$ series on the right-hand side
of (\ref{1e3P2}) gives a relation between two terminating
${}_4F_3(4)$ series:

\begin{Proposition}
\label{3P3}
If $n$ is a nonnegative integer, 
the following relation between two 
terminating ${}_4F_3(4)$ series holds:
\begin{eqnarray}
\label{1e3P3}
&&{}_4F_3 \left( \left. 
{\displaystyle -n,\frac{a}{2},\frac{a+1}{2},b
\atop \displaystyle a,1+a-c,c}\right| 
4\right)\\
&&=\frac{(-1)^n(b)_n}{(1+a-c)_n}
{}_4F_3 \left( \left. 
{\displaystyle -n,\frac{c-b-n}{2},\frac{c-b-n+1}{2},c-a-n
\atop \displaystyle c-b-n,1-b-n,c}\right| 
4\right).\nonumber
\end{eqnarray}  
\end{Proposition}

\begin{proof}
In Proposition \ref{3P2}, we reverse the order of summation 
in the ${}_4F_3(1/4)$ series on the right-hand side 
according to Equation
(\ref{rs}). The right-hand side 
in (\ref{1e3P2}) thus becomes
\begin{eqnarray}
\label{1ep3P3}
&&\frac{(c-b)_n}{(c)_n}
{}_4F_3 \left( \left. 
{\displaystyle -n,1-c-n,b,1+b-c
\atop \displaystyle 1+a-c,\frac{1+b-c-n}{2},\frac{2+b-c-n}{2}}\right| 
\frac{1}{4}\right)\\
&&=\frac{(c-b)_n(1-c-n)_n(b)_n(1+b-c)_n\left(-\frac{1}{4}\right)_n}
{(c)_n(1+a-c)_n\left(\frac{1+b-c-n}{2}\right)_n\left(\frac{2+b-c-n}{2}\right)_n}\nonumber\\
&&\times
{}_4F_3 \left( \left. 
{\displaystyle -n,\frac{c-b-n}{2},\frac{c-b-n+1}{2},c-a-n
\atop \displaystyle c-b-n,1-b-n,c}\right| 
4\right)\nonumber\\
&&=\frac{(c-b)_n(b)_n(1+b-c)_n}
{(1+a-c)_n(1+b-c-n)_{2n}}\nonumber\\
&&\times
{}_4F_3 \left( \left. 
{\displaystyle -n,\frac{c-b-n}{2},\frac{c-b-n+1}{2},c-a-n
\atop \displaystyle c-b-n,1-b-n,c}\right| 
4\right)\nonumber\\
&&=\frac{(-1)^n(b)_n}{(1+a-c)_n}
{}_4F_3 \left( \left. 
{\displaystyle -n,\frac{c-b-n}{2},\frac{c-b-n+1}{2},c-a-n
\atop \displaystyle c-b-n,1-b-n,c}\right| 
4\right)\nonumber
\end{eqnarray}
and the result follows.
\end{proof}

We define now
\begin{equation}
\label{Tdef}
T_n(a,b,c)
=(1+a-c)_n(c)_n\,
{}_4F_3 \left( \left. 
{\displaystyle -n,\frac{a}{2},\frac{a+1}{2},b
\atop \displaystyle a,1+a-c,c}\right| 
4\right).\\
\end{equation}

The function $T_n$ defined above has the trivial invariance
\begin{equation}
\label{Ti1}
T_n(a,b,c)=T_n(a,b,1+a-c).
\end{equation}

Proposition \ref{3P3} gives us a notrivial invariance for $T_n$:
\begin{equation}
\label{Ti2}
T_n(a,b,c)=T_n(c-b-n,c-a-n,c).
\end{equation}

The invariances (\ref{Ti1}) and (\ref{Ti2}) generate
an invariance group $G$ for the function $T_n$.
The next
theorem lists all six 
resulting relations in the invariance group $G$ and
describes $G$ as isomorphic
to the symmetric group $S_3$:

\begin{Theorem}
\label{3T1}
Let $G$ be the invariance group for the function
$T_n(a,b,c)$ generated by 
(\ref{Ti1}) and (\ref{Ti2}) above. Then $G$ is isomorphic
to the symmetric group $S_3$ of order 6. Furthermore,
the resulting six invariances for $T_n(a,b,c)$ are given by:
\begin{eqnarray}
\label{TI1}
&&T_n(a,b,c)=T_n(a,b,c),\\
\label{TI2}
&&T_n(a,b,c)=T_n(a,b,1+a-c),\\
\label{TI3}
&&T_n(a,b,c)=T_n(c-b-n,c-a-n,c),\\
\label{TI4}
&&T_n(a,b,c)=T_n(c-b-n,c-a-n,1-b-n),\\
\label{TI5}
&&T_n(a,b,c)=T_n(1+a-b-c-n,1-c-n,1+a-c),\\
\label{TI6}
&&T_n(a,b,c)=T_n(1+a-b-c-n,1-c-n,1-b-n).
\end{eqnarray}
\end{Theorem}

\begin{proof}
By combining (\ref{Ti1}) and (\ref{Ti2}) in all 
possible ways, it is straight-forward to check that we obtain
the relations (\ref{TI1})--(\ref{TI6}).

It is well-known that
there are two groups of order 6: the cyclic group 
$\mathbb{Z}_6$ and the symmetric group $S_3$.
We directly compute the orders of the relations 
(\ref{TI1})--(\ref{TI6}) in $G$ to be
1, 2, 2, 3, 3, and 2, respectively.
Based on the orders of these elements,
in particular the lack of an element of order 6,
the invariance group $G$ must be isomorphic to 
the symmetric group $S_3$.
\end{proof}

We next reparameterize the function $T_n$ to
emphasize the $S_3$-symmetry further:

\begin{Theorem}
\label{3T2}
Define the function $U_n$ by
\begin{equation}
\label{Udef}
U_n(x,y,z)
=T_n\left(x-y-z,\frac{1+3x-y-z-2n}{2},\frac{1+x+y-3z}{2}\right).
\end{equation}
Then $U_n$ is a symmetric function in its parameters
$x,y,z$, i.e. $U_n$ is invariant under any of the six possible permutations
of $x,y,z$. Furthermore, the six invariances of $U_n$ given by
\begin{eqnarray}
\label{UI1}
&&U_n(x,y,z)=U_n(x,y,z),\\
\label{UI2}
&&U_n(x,y,z)=U_n(x,z,y),\\
\label{UI3}
&&U_n(x,y,z)=U_n(y,x,z),\\
\label{UI4}
&&U_n(x,y,z)=U_n(y,z,x),\\
\label{UI5}
&&U_n(x,y,z)=U_n(z,x,y),\\
\label{UI6}
&&U_n(x,y,z)=U_n(z,y,x)
\end{eqnarray}
correspond to the invariances 
(\ref{TI1})--(\ref{TI6}), respectively,
of $T_n(a,b,c)$ given in Theorem \ref{3T1}.
\end{Theorem}

\begin{proof}
The result 
in this theorem
follows from a direct check.
\end{proof}

\section{Relations for the terminating
${}_3F_2(4)$ series}

In this section, we study consequences of the relations 
from the previous section. We derive
as special cases relations between terminating
${}_3F_2(4)$ hypergeometric series.

If we let $b=a$ in Proposition \ref{3P3},
we obtain a relation between two 
terminating ${}_3F_2(4)$ series:

\begin{eqnarray}
\label{1e4c1}
&&{}_3F_2 \left( \left. 
{\displaystyle -n,\frac{a}{2},\frac{a+1}{2}
\atop \displaystyle 1+a-c,c}\right| 
4\right)\\
&&=\frac{(-1)^n(a)_n}{(1+a-c)_n}
{}_3F_2 \left( \left. 
{\displaystyle -n,\frac{c-a-n}{2},\frac{c-a-n+1}{2}
\atop \displaystyle 1-a-n,c}\right| 
4\right).\nonumber
\end{eqnarray}  

Let us define the function
\begin{equation}
\label{Ttdef}
\widetilde{T}_n(a,c)
=(1+a-c)_n(c)_n\,
{}_3F_2 \left( \left. 
{\displaystyle -n,\frac{a}{2},\frac{a+1}{2}
\atop \displaystyle 1+a-c,c}\right| 
4\right).
\end{equation}

A trivial relation for $\widetilde{T}_n$ is 
\begin{equation}
\label{Tti1}
\widetilde{T}_n(a,c)=\widetilde{T}_n(a,1+a-c).
\end{equation}

Furthermore, Equation (\ref{1e4c1}) gives the nontrivial relation
\begin{equation}
\label{Tti2}
\widetilde{R}_n(a,c)=\widetilde{R}_n(c-a-n,c).
\end{equation}

Just like the 
invariance group for the 
function $T_n(a,b,c)$ for the terminating
${}_4F_3(4)$ series, the invariance group for
$\widetilde{T}_n(a,c)$ is isomorphic to the symmetric
group $S_3$ as well. 
In fact, the invariance relations 
for $\widetilde{T}_n$
follow from the
invariance relations (\ref{TI1})--(\ref{TI6})
for $T$ upon setting $b=a$ in each 
of the latter relations.
Below is a list of the six invariances for 
$\widetilde{T}_n(a,c)$:
\begin{eqnarray}
\label{TtI1}
&&\widetilde{T}_n(a,c)=\widetilde{T}_n(a,c),\\
\label{TtI2}
&&\widetilde{T}_n(a,c)=\widetilde{T}_n(a,1+a-c),\\
\label{TtI3}
&&\widetilde{T}_n(a,c)=\widetilde{T}_n(c-a-n,c),\\
\label{TtI4}
&&\widetilde{T}_n(a,c)=\widetilde{T}_n(c-a-n,1-a-n),\\
\label{TtI5}
&&\widetilde{T}_n(a,c)=\widetilde{T}_n(1-c-n,1+a-c),\\
\label{TtI6}
&&\widetilde{T}_n(a,c)=\widetilde{T}_n(1-c-n,1-a-n).
\end{eqnarray}
Invariances (\ref{TtI1})--(\ref{TtI6})
correspond to invariances (\ref{TI1})--(\ref{TI6}),
respectively.

If we reparameterize $\widetilde{T}_n(a,c)$ by
\begin{equation}
\label{Utdef}
\widetilde{U}_n(x,y,z)
=\widetilde{T}_n\left(\frac{1+2x-y-z-2n}{3},\frac{2+x+y-2z-n}{3}\right),
\end{equation}
then $\widetilde{U}_n$ is invariant under all six permutations
of $x,y,z$. In fact, the invariances
\begin{eqnarray}
\label{UtI1}
&&\widetilde{U}_n(x,y,z)=\widetilde{U}_n(x,y,z),\\
\label{UtI2}
&&\widetilde{U}_n(x,y,z)=\widetilde{U}_n(x,z,y),\\
\label{UtI3}
&&\widetilde{U}_n(x,y,z)=\widetilde{U}_n(y,x,z),\\
\label{UtI4}
&&\widetilde{U}_n(x,y,z)=\widetilde{U}_n(y,z,x),\\
\label{UtI5}
&&\widetilde{U}_n(x,y,z)=\widetilde{U}_n(z,x,y),\\
\label{UtI6}
&&\widetilde{U}_n(x,y,z)=\widetilde{U}_n(z,y,x)
\end{eqnarray}
correspond to the invariances 
(\ref{TtI1})--(\ref{TtI6}), respectively,
of $\widetilde{T}_n(a,c)$.

We next let $b=1+a-c$ in Proposition \ref{3P3}.
We obtain the following relation between two 
terminating ${}_3F_2(4)$ series:
\begin{eqnarray}
\label{1e4c2}
&&{}_3F_2 \left( \left. 
{\displaystyle -n,\frac{a}{2},\frac{a+1}{2}
\atop \displaystyle a,c}\right| 
4\right)\\
&&=(-1)^n
{}_3F_2 \left( \left. 
{\displaystyle -n,\frac{2c-a-n-1}{2},\frac{2c-a-n}{2}
\atop \displaystyle 2c-a-n-1,c}\right| 
4\right).\nonumber
\end{eqnarray}  

We note that relation (\ref{1e4c2}) is different from relation 
(\ref{1e4c1}) as
the two ${}_3F_2(4)$ series 
in (\ref{1e4c2}) are different in type from the two
${}_3F_2(4)$ series in (\ref{1e4c1}).

If we define
\begin{equation}
\label{Qdef}
Q_n(a,c)
=(1+a-c)_n(c)_n\,
{}_3F_2 \left( \left. 
{\displaystyle -n,\frac{a}{2},\frac{a+1}{2}
\atop \displaystyle a,c}\right| 
4\right),
\end{equation}
then (\ref{1e4c2}) leads to
\begin{equation}
\label{Qinv}
Q_n(a,c)=Q_n(2c-a-n-1,c),
\end{equation}
which is a relation of order 2. The function
$Q_n(a,c)$ does not have any trivial relations besides
the identity, and thus the invariance group for $Q_n(a,c)$
is isomorphic to the symmetric group $S_2$ of order 2.
In fact, if we define
\begin{equation}
\label{Wdef}
W_n(x,y)=Q_n\left(x,\frac{1+x+y+n}{2}\right),
\end{equation}
the nontrivial relation (\ref{Qinv}) for $Q_n(a,c)$
can be written as 
\begin{equation}
\label{Winv}
W_n(x,y)=W_n(y,x).
\end{equation}

There is
one more 
nontrivial relation for the 
series
${}_3F_2 \left( \left. 
{\displaystyle -n,\frac{a}{2},\frac{a+1}{2}
\atop \displaystyle a,c}\right| 
4\right)$ 
that
can be obtained from the relations
(\ref{TI1})--(\ref{TI6}) in Theorem \ref{3T1}:
letting $b=1+a-c$ in (\ref{TI5}) (or in (\ref{TI6})), we have
\begin{eqnarray}
\label{2e4c2}
&&{}_3F_2 \left( \left. 
{\displaystyle -n,\frac{a}{2},\frac{a+1}{2}
\atop \displaystyle a,c}\right| 
4\right)\\
&&=\frac{(-1)^n(1+a-c)_n}{(c)_n}
{}_3F_2 \left( \left. 
{\displaystyle -\frac{n}{2},\frac{1-n}{2},1-c-n
\atop \displaystyle c-a-n,1+a-c}\right| 
4\right).\nonumber
\end{eqnarray}  

\begin{Remark}
\label{4R2}
Let $a=1$ in Equation (\ref{2e4c2}) above. 
We obtain the relation
\begin{eqnarray}
\label{1e4R2}
&&{}_2F_1 \left( \left. 
{\displaystyle -n,\frac{1}{2}
\atop \displaystyle c}\right| 
4\right)\\
&&=\frac{(-1)^n(2-c)_n}{(c)_n}
{}_3F_2 \left( \left. 
{\displaystyle -\frac{n}{2},\frac{1-n}{2},1-c-n
\atop \displaystyle c-n-1,2-c}\right| 
4\right).\nonumber
\end{eqnarray}  
The special case 
$c=1+\gamma+\lceil\frac{n}{2}\rceil$,
where $\gamma$ is an integer and 
$\lceil x \rceil$ denotes the smallest integer greater than
or equal to the real number $x$,
in 
the ${}_2F_1(4)$ series on the left-hand side of
(\ref{1e4R2}) coincides with the special
case $\alpha=0$ of
the function $\Omega_n(\alpha,\gamma)$ defined
in \cite[Eq.\ (1)]{Chu1}.
If we take the limit
as $c \to 1+\gamma+\lceil\frac{n}{2}\rceil$ in 
(\ref{1e4R2}), we can find the corresponding values
of $\Omega_n(0,\gamma)$.
For example, the specific
values of $\Omega_n(0,\gamma)$ given in
\cite{Chu1} for $\gamma=-2,-1,0,1,2$ can all
be obtained as limits of our relation (\ref{1e4R2})
as just described. 
\end{Remark}

\begin{Remark}
\label{4R1}
The special case 
$c=\gamma+a+\lceil\frac{n}{2}\rceil$,
where $\gamma$ is an integer,
in the ${}_3F_2(4)$ series on the left-hand side 
of Equation (\ref{2e4c2})
coincides with the special case $\alpha=0$ of
the function $\Omega_{\alpha,\gamma}^n(a)$ defined
in \cite[Eq.\ (1.3)]{ChenChu1}. In fact, by taking the limit
as $c \to \gamma+a+\lceil\frac{n}{2}\rceil$ in 
(\ref{2e4c2}), we can find the corresponding values
of $\Omega_{0,\gamma}^n(a)$.
For example, the specific
values of $\Omega_{0,\gamma}^n(a)$ given in
\cite{ChenChu1} for $\gamma=-2,-1,0,1,2$ can all
be obtained as limits of our relation (\ref{2e4c2})
as just described.
\end{Remark}

For the hypergeometric series 
${}_{p}F_q \left( \left.{\displaystyle a_1,a_2,\ldots,a_{p}
\atop \displaystyle b_1,b_2,\ldots,b_q}\right| z\right)$,
we let the expression
$\left[{}_{p}F_q \left( \left.{\displaystyle a_1,a_2,\ldots,a_{p}
\atop \displaystyle b_1,b_2,\ldots,b_q}\right| z\right)\right]_n$
denote the sum of the first $n+1$ terms of the series, i.e.
\begin{equation}
\label{parsum}
\left[{}_{p}F_q \left( \left.{\displaystyle a_1,a_2,\ldots,a_{p}
\atop \displaystyle b_1,b_2,\ldots,b_q}\right| z\right)\right]_n =
\sum_{k=0}^{n} \frac{(a_1)_k(a_2)_k \cdots
(a_{p})_k}{k!(b_1)_k(b_2)_k \cdots (b_q)_k}z^k.
\end{equation}

Now let $c=1+a+n$ in Proposition \ref{3P3}.
We obtain the 
following
curious result regarding the sum of the first 
$n+1$ terms of the divergent series
${}_3F_2 \left( \left. 
{\displaystyle \frac{a}{2},\frac{a+1}{2},b
\atop \displaystyle a,1+a+n}\right| 
4\right)$:
\begin{eqnarray}
\label{1e4P3}
&&\left[{}_3F_2 \left( \left. 
{\displaystyle \frac{a}{2},\frac{a+1}{2},b
\atop \displaystyle a,1+a+n}\right| 
4\right)\right]_n\\
&&=\frac{(b)_n}{n!}
{}_4F_3 \left( \left. 
{\displaystyle -n,\frac{1+a-b}{2},\frac{2+a-b}{2},1
\atop \displaystyle 1+a-b,1-b-n,1+a+n}\right| 
4\right).\nonumber
\end{eqnarray} 

Letting $c=1+a+n$ in (\ref{TI5}) (or in (\ref{TI6}))
and then combining the result with (\ref{1e4P3}),
we obtain a formula for 
the sum of the first $n+1$ terms of the divergent series
${}_3F_2 \left( \left. 
{\displaystyle -\frac{b}{2}-n,\frac{1-b}{2}-n,-a-2n
\atop \displaystyle -b-2n,1-b-n}\right| 
4\right)$:
\begin{eqnarray}
\label{2e4P4}
&&\left[{}_3F_2 \left( \left. 
{\displaystyle -\frac{b}{2}-n,\frac{1-b}{2}-n,-a-2n
\atop \displaystyle -b-2n,1-b-n}\right| 
4\right)\right]_n\\
&&=\frac{(-1)^n(1+a)_{2n}}{n!(1+a)_n}
{}_4F_3 \left( \left. 
{\displaystyle -n,\frac{1+a-b}{2},\frac{2+a-b}{2},1
\atop \displaystyle 1+a-b,1-b-n,1+a+n}\right| 
4\right).\nonumber
\end{eqnarray}  

\section{Relations for the terminating
${}_4F_3(1/4)$ series}

We explore relations for terminating ${}_4F_3(1/4)$ hypergeometric
series in this section and describe the structure of those
relations. The relations for the terminating  ${}_4F_3(1/4)$
series correspond to the series reversals of the relations for the
terminating ${}_4F_3(4)$ series in Section 3.

\begin{Proposition}
\label{5P1}
If $n$ is a nonnegative integer, 
the following relation between two 
terminating ${}_4F_3(1/4)$ series holds:
\begin{eqnarray}
\label{1e5P1}
&&{}_4F_3 \left( \left. 
{\displaystyle -n,a,a-c-n,c
\atop \displaystyle \frac{a-n}{2},\frac{1+a-n}{2},b}\right| 
\frac{1}{4}\right)\\
&&=\frac{(1+c-a)_n(b-c)_n}{(1-a)_n(b)_n}\nonumber\\
&&\times
{}_4F_3 \left( \left. 
{\displaystyle -n,1+c-b,1-b-n,c
\atop \displaystyle \frac{1+c-b-n}{2},\frac{2+c-b-n}{2},1+c-a}\right| 
\frac{1}{4}\right).\nonumber
\end{eqnarray}  
\end{Proposition}

\begin{proof}
Reverse the order of summation
in the two ${}_4F_3(4)$ series
in Proposition \ref{3P3}
according to (\ref{rs}), and then re-label
$1-a-n,1-b-n$, and $1-c-n$ with 
$a,b$, and $c$, respectively.
\end{proof}

Let
\begin{equation}
\label{Rdef}
R_n(a,b,c)
=(1-a)_n(b)_n\,
{}_4F_3 \left( \left. 
{\displaystyle -n,a,a-c-n,c
\atop \displaystyle \frac{a-n}{2},\frac{1+a-n}{2},b}\right| 
\frac{1}{4}\right).\\
\end{equation}

The function $R_n$ has the trivial invariance
\begin{equation}
\label{Ri1}
R_n(a,b,c)=R_n(a,b,a-c-n).
\end{equation}

Proposition \ref{5P1} gives us a notrivial invariance for $R_n$:
\begin{equation}
\label{Ri2}
R_n(a,b,c)=R_n(1+c-b,1+c-a,c).
\end{equation}

The invariance group for the function
$R_n(a,b,c)$ generated by (\ref{Ri1})
and (\ref{Ri2}) is, just like the invariance
group for $T_n(a,b,c)$, isomorphic to the symmetric group $S_3$
of order 6. The invariance relations for $R_n$ are given by
\begin{eqnarray}
\label{RI1}
&&R_n(a,b,c)=R_n(a,b,c),\\
\label{RI2}
&&R_n(a,b,c)=R_n(a,b,a-c-n),\\
\label{RI3}
&&R_n(a,b,c)=R_n(1+c-b,1+c-a,c),\\
\label{RI4}
&&R_n(a,b,c)=R_n(1+c-b,1+c-a,1-b-n),\\
\label{RI5}
&&R_n(a,b,c)=R_n(1+a-b-c-n,1-c-n,a-c-n),\\
\label{RI6}
&&R_n(a,b,c)=R_n(1+a-b-c-n,1-c-n,1-b-n),
\end{eqnarray}
and these 
invariance relations
correspond to the summation reversals of
(\ref{TI1})--(\ref{TI6}), respectively.

Next, we reparameterize $R_n$ by defining $V_n$
according to
\begin{equation}
\label{Vdef}
V_n(x,y,z)=R_n\left(
x-y-z,\frac{2+3x-y-z-n}{2},\frac{x+y-3z-n}{2}\right).
\end{equation}
Then $V_n(x,y,z)$ is invariant under all six permutations
of $x,y,z$. Furthermore, the invariances
\begin{eqnarray}
\label{VI1}
&&V_n(x,y,z)=V_n(x,y,z),\\
\label{VI2}
&&V_n(x,y,z)=V_n(x,z,y),\\
\label{VI3}
&&V_n(x,y,z)=V_n(y,x,z),\\
\label{VI4}
&&V_n(x,y,z)=V_n(y,z,x),\\
\label{VI5}
&&V_n(x,y,z)=V_n(z,x,y),\\
\label{VI6}
&&V_n(x,y,z)=V_n(z,y,x)
\end{eqnarray}
correspond to the invariances 
(\ref{RI1})--(\ref{RI6}), respectively,
of $R_n(a,b,c)$.

\section{Relations for the terminating
${}_3F_2(1/4)$ series}

In this final section, we derive some consequences of the relations for
the ${}_4F_3(1/4)$ hypergeometric series from the previous section.
We obtain relations for terminating ${}_3F_2(1/4)$ series. We note that
the relations for the terminating ${}_3F_2(1/4)$ series in this section
correspond to the series reversals of the relations for the
terminating ${}_3F_2(4)$ series in Section 4.

Let $b=a$ in Proposition \ref{5P1}. We obtain the following relation
between two terminating ${}_3F_2(1/4)$ series:
\begin{eqnarray}
\label{1e5S1}
&&{}_3F_2 \left( \left. 
{\displaystyle -n,a-c-n,c
\atop \displaystyle \frac{a-n}{2},\frac{1+a-n}{2}}\right| 
\frac{1}{4}\right)\\
&&=\frac{(1+c-a)_n(a-c)_n}{(1-a)_n(a)_n}\nonumber\\
&&\times
{}_3F_2 \left( \left. 
{\displaystyle -n,1-a-n,c
\atop \displaystyle \frac{1+c-a-n}{2},\frac{2+c-a-n}{2}}\right| 
\frac{1}{4}\right).\nonumber
\end{eqnarray}  

Define the function
\begin{equation}
\label{Rtdef}
\widetilde{R}_n(a,c)
=(1-a)_n(a)_n\,
{}_3F_2 \left( \left. 
{\displaystyle -n,a-c-n,c
\atop \displaystyle \frac{a-n}{2},\frac{1+a-n}{2}}\right| 
\frac{1}{4}\right).\\
\end{equation}

A trivial relation for $\widetilde{R}_n$ is 
\begin{equation}
\label{Rti1}
\widetilde{R}_n(a,c)=\widetilde{R}_n(a,a-c-n).
\end{equation}

Furthermore, Equation (\ref{1e5S1}) gives the nontrivial relation
\begin{equation}
\label{Rti2}
\widetilde{R}_n(a,c)=\widetilde{R}_n(1+c-a,c).
\end{equation}

The invariance group for the function 
$\widetilde{R}_n$ is isomorphic to the symmetric
group $S_3$. The six invariances of $\widetilde{R}_n$
are given by
\begin{eqnarray}
\label{RtI1}
&&\widetilde{R}_n(a,c)=\widetilde{R}_n(a,c),\\
\label{RtI2}
&&\widetilde{R}_n(a,c)=\widetilde{R}_n(a,a-c-n),\\
\label{RtI3}
&&\widetilde{R}_n(a,c)=\widetilde{R}_n(1+c-a,c),\\
\label{RtI4}
&&\widetilde{R}_n(a,c)=\widetilde{R}_n(1+c-a,1-a-n),\\
\label{RtI5}
&&\widetilde{R}_n(a,c)=\widetilde{R}_n(1-c-n,a-c-n),\\
\label{RtI6}
&&\widetilde{R}_n(a,c)=\widetilde{R}_n(1-c-n,1-a-n),
\end{eqnarray}
and they correspond to setting $b=a$ in 
relations (\ref{RI1})--(\ref{RI6}), respectively,
for the function $R_n(a,b,c)$.

If we reparameterize $\widetilde{R}_n(a,c)$ by
\begin{equation}
\label{Vtdef}
\widetilde{V}_n(x,y,z)
=\widetilde{R}_n\left(\frac{2+2x-y-z-n}{3},\frac{1+x+y-2z-2n}{3}\right),
\end{equation}
then $\widetilde{V}_n$ is invariant under all six permutations
of $x,y,z$. The invariances
\begin{eqnarray}
\label{VtI1}
&&\widetilde{V}_n(x,y,z)=\widetilde{V}_n(x,y,z),\\
\label{VtI2}
&&\widetilde{V}_n(x,y,z)=\widetilde{V}_n(x,z,y),\\
\label{VtI3}
&&\widetilde{V}_n(x,y,z)=\widetilde{V}_n(y,x,z),\\
\label{VtI4}
&&\widetilde{V}_n(x,y,z)=\widetilde{V}_n(y,z,x),\\
\label{VtI5}
&&\widetilde{V}_n(x,y,z)=\widetilde{V}_n(z,x,y),\\
\label{VtI6}
&&\widetilde{V}_n(x,y,z)=\widetilde{V}_n(z,y,x)
\end{eqnarray}
correspond to the invariances 
(\ref{RtI1})--(\ref{RtI6}), respectively,
of $\widetilde{R}_n(a,c)$.

Next, let us set $b=a-c-n$ in Proposition \ref{5P1}.
We obtain the following relation between two 
terminating ${}_3F_2(1/4)$ series:
\begin{eqnarray}
\label{1e5S2}
&&{}_3F_2 \left( \left. 
{\displaystyle -n,a,c
\atop \displaystyle \frac{a-n}{2},\frac{1+a-n}{2}}\right| 
\frac{1}{4}\right)\\
&&=\frac{(1+2c-a)_n}{(1-a)_n}
{}_3F_2 \left( \left. 
{\displaystyle -n,1+2c-a+n,c
\atop \displaystyle \frac{1+2c-a}{2},\frac{2+2c-a}{2}}\right| 
\frac{1}{4}\right).\nonumber
\end{eqnarray}  

We define the function
\begin{equation}
\label{Mdef}
M_n(a,c)
=(1-a)_n(1+c-a)_n\,
{}_3F_2 \left( \left. 
{\displaystyle -n,a,c
\atop \displaystyle \frac{a-n}{2},\frac{1+a-n}{2}}\right| 
\frac{1}{4}\right).
\end{equation}

We note that
Equation (\ref{1e5S2}) gives us
\begin{equation}
\label{Minv}
M_n(a,c)=M_n(1+2c-a+n,c),
\end{equation}
which is a relation of order 2. Since the function
$M_n(a,c)$ does not have any trivial relations besides
the identity, the invariance group for $M_n(a,c)$
is isomorphic to the symmetric group $S_2$ of order 2.
Furthermore, if we define
\begin{equation}
\label{Ldef}
L_n(x,y)=M_n\left(x,\frac{x+y-n-1}{2}\right),
\end{equation}
the nontrivial relation (\ref{Minv}) for $M_n(a,c)$
can be written as 
\begin{equation}
\label{Linv}
L_n(x,y)=L_n(y,x).
\end{equation}

The other nontrivial relation for the series
${}_3F_2 \left( \left. 
{\displaystyle -n,a,c
\atop \displaystyle \frac{a-n}{2},\frac{1+a-n}{2}}\right| 
\displaystyle{\frac{1}{4}}\right)$
is given in the next proposition and
it has two cases depending on whether
$n$ is even or odd:

\begin{Proposition}
\label{5P2}
If $n$ is a nonnegative integer, then
the following two relations hold:
\begin{eqnarray}
\label{1e5P2}
&&{}_3F_2 \left( \left. 
{\displaystyle -2n,a,c
\atop \displaystyle \frac{a}{2}-n,\frac{a+1}{2}-n}\right| 
\frac{1}{4}\right)\\
&&=\frac{(-1)^n(2n)!(c)_n}{n!(1-a)_{2n}}\nonumber\\
&&\times {}_3F_2 \left( \left. 
{\displaystyle -n,1+c-a+n,a-c-n
\atop \displaystyle \frac{1}{2},1-c-n}\right| 
\frac{1}{4}\right)\nonumber
\end{eqnarray}  
and
\begin{eqnarray}
\label{2e5P2}
&&{}_3F_2 \left( \left. 
{\displaystyle -2n-1,a,c
\atop \displaystyle \frac{a-1}{2}-n,\frac{a}{2}-n}\right| 
\frac{1}{4}\right)\\
&&=\frac{(-1)^n(1+c-a+n)(2n+1)!(c)_n}{n!(1-a)_{2n+1}}\nonumber\\
&&\times {}_3F_2 \left( \left. 
{\displaystyle -n,2+c-a+n,a-c-n
\atop \displaystyle \frac{3}{2},1-c-n}\right| 
\frac{1}{4}\right).\nonumber
\end{eqnarray}  
\end{Proposition}

\begin{proof}
To prove (\ref{1e5P2}), we start by replacing $n$ with $2n$ in relation
(\ref{RI5}) (or (\ref{RI6}))
for the terminating ${}_4F_3(1/4)$ series.
The resulting relation can be written as
\begin{eqnarray}
\label{1ep5P2}
&&{}_4F_3 \left( \left. 
{\displaystyle -2n,a,a-c-2n,c
\atop \displaystyle \frac{a}{2}-n,\frac{a+1}{2}-n,b}\right| 
\frac{1}{4}\right)\\
&&=\frac{(c)_{2n}(1+a-b-c-4n)_{2n}}{(1-a)_{2n}(b)_{2n}}\nonumber\\
&&\times
{}_4F_3 \left( \left. 
{\displaystyle -2n,1+a-b-c-2n,1-b-2n,a-c-2n
\atop \displaystyle \frac{1+a-b-c-4n}{2},\frac{2+a-b-c-4n}{2},1-c-2n}\right| 
\frac{1}{4}\right).\nonumber
\end{eqnarray}  

We now let $b \to a-c-2n$ in (\ref{1ep5P2}). The left-hand side of (\ref{1ep5P2})
turns into the left-hand side of (\ref{1e5P2}). For the 
right-hand side of (\ref{1ep5P2}), we have
\begin{eqnarray}
\label{2ep5P2}
&&\lim_{b \to a-c-2n} \left(
\frac{(c)_{2n}(1+a-b-c-4n)_{2n}}{(1-a)_{2n}(b)_{2n}}\right.\\
&&\left.\times
{}_4F_3 \left( \left. 
{\displaystyle -2n,1+a-b-c-2n,1-b-2n,a-c-2n
\atop \displaystyle \frac{1+a-b-c-4n}{2},\frac{2+a-b-c-4n}{2},1-c-2n}\right| 
\frac{1}{4}\right)\right)\nonumber\\
&&=\frac{(c)_{2n}}{(1-a)_{2n}(1+c-a)_{2n}}
\lim_{b \to a-c-2n} \Bigg(
(1+a-b-c-4n)_{2n}\nonumber\\
&&\left.\times
\sum_{k=0}^{2n}  
\frac{(-2n)_k(1+a-b-c-2n)_k(1-b-2n)_k(a-c-2n)_k}
{k!(1+a-b-c-4n)_{2k}(1-c-2n)_k}
\right)\nonumber\\
&&=\frac{(c)_{2n}}{(1-a)_{2n}(1+c-a)_{2n}}\nonumber\\
&&\times \lim_{b \to a-c-2n}
\sum_{k=n}^{2n}  
\frac{(-2n)_k(1+a-b-c-2n)_k(1-b-2n)_k(a-c-2n)_k}
{k!(1+a-b-c-2n)_{2k-2n}(1-c-2n)_k}\nonumber\\
&&=\frac{(c)_{2n}}{(1-a)_{2n}(1+c-a)_{2n}}\nonumber\\
&&\times 
\sum_{k=n}^{2n}  
\frac{(-2n)_k(1+c-a)_k(a-c-2n)_k}
{(1)_{2k-2n}(1-c-2n)_k}.\nonumber
\end{eqnarray} 

After the change of index $k=n+m$, the right-hand side above becomes
\begin{eqnarray}
\label{3ep5P2}
&&\frac{(c)_{2n}}{(1-a)_{2n}(1+c-a)_{2n}}\\
&&\times 
\sum_{m=0}^{n}  
\frac{(-2n)_{n+m}(1+c-a)_{n+m}(a-c-2n)_{n+m}}
{(1)_{2m}(1-c-2n)_{n+m}}\nonumber\\
&&=\frac{(c)_{2n}(-2n)_n(1+c-a)_n(a-c-2n)_n}
{(1-a)_{2n}(1+c-a)_{2n}(1-c-2n)_n}\nonumber\\
&&\times 
{}_3F_2 \left( \left. 
{\displaystyle -n,1+c-a+n,a-c-n
\atop \displaystyle \frac{1}{2},1-c-n}\right| 
\frac{1}{4}\right),\nonumber
\end{eqnarray} 
which simplifies to the
right-hand side of Equation (\ref{1e5P2}),
thus proving (\ref{1e5P2}).

The proof of (\ref{2e5P2}) follows similar lines: 
we start by replacing $n$ with $2n+1$ in relation
(\ref{RI5}) (or (\ref{RI6}))
for the terminating ${}_4F_3(1/4)$ series
and then let $b \to a-c-2n-1$ 
in the resulting relation
in the same manner as 
above.
\end{proof}

As a final result, 
we obtain an equation for the sum of the first 
$n+1$ terms of a certain nonterminating 
${}_3F_2(1/4)$ hypergeometric series.
We do this by letting
$b=-n$ in Proposition \ref{5P1}.
We consequently obtain an equation for the sum of the first
$n+1$ terms of the nonterminating series
${}_3F_2 \left( \left. 
{\displaystyle a,a-c-n,c
\atop \displaystyle \frac{a-n}{2},\frac{1+a-n}{2}}\right| 
\displaystyle{\frac{1}{4}}\right)$:
\begin{eqnarray}
\label{1e5S3}
&&\left[{}_3F_2 \left( \left. 
{\displaystyle a,a-c-n,c
\atop \displaystyle \frac{a-n}{2},\frac{1+a-n}{2}}\right| 
\frac{1}{4}\right)\right]_n\\
&&=\frac{(1+c-a)_n(1+c)_n}{n!(1-a)_n}
{}_4F_3 \left( \left. 
{\displaystyle -n,1+c+n,c,1
\atop \displaystyle \frac{1+c}{2},\frac{2+c}{2},1+c-a}\right| 
\frac{1}{4}\right).\nonumber
\end{eqnarray}

\end{document}